\documentclass{amsart}%
\usepackage{amsfonts}
\usepackage{amsmath}
\usepackage{amssymb}
\usepackage{graphicx}%
\providecommand{\U}[1]{\protect \rule{.1in}{.1in}}
\theoremstyle{plain}
\newtheorem{theorem}{Theorem}[section]

\newtheorem{corollary}[theorem]{Corollary}

\newtheorem{example}[theorem]{Example}

\newtheorem{lemma}[theorem]{Lemma}

\newtheorem{proposition}[theorem]{Proposition}
\newtheorem{remark}[theorem]{Remark}

\numberwithin{equation}{section}
\begin{document}
\title[AOP mappings]{A note on linear approximately orthogonality preserving mappings}
\author{Ye Zhang}
\address{Department of Mathematics, University of New Hampshire, Durham, NH 03824, U.S.A.}
\email{yjg2@wildcats.unh.edu}
\author{Yanni Chen}
\address{Department of Mathematics, University of New Hampshire, Durham, NH 03824, U.S.A.}
\email{yanni.chen@unh.edu}
\author{Don Hadwin}
\address{Department of Mathematics, University of New Hampshire, Durham, NH 03824, U.S.A.}
\email{don@unh.edu}
\urladdr{http://euclid.unh.edu/\symbol{126}don}
\author{Liang Kong}
\address{Institute of Applied Mathematics,Shangluo University, Shangluo 726000, P. R. China}
\email{kongliang2005@163.com}
\thanks{Research supported in part by a Dissertation Year Fellowship form University of New Hampshire.}

\begin{abstract}
In this paper, linear $\varepsilon$-orthogonality preserving mappings are
studied. We define $\hat{\varepsilon}\left(T\right) $ as the smallest
$\varepsilon$ for which $T$ is $\varepsilon$-orthogonality preserving, and
then derive an exact formula for $\hat{\varepsilon}\left(  T\right)  $ in
terms of $\left \Vert T\right \Vert $ and the minimum modulus $m\left(
T\right)  $ of $T$. We see that $\varepsilon$-orthogonality preserving
mappings (for some $\varepsilon<1$) are exactly the operators that are bounded
from below. We improve an upper bounded in the stability equation given in [7,
Theorem 2.3], which was thought to be sharp.
\end{abstract}

\subjclass[2010]{46C05, 47L05, 47B99.}
\keywords{approximate orthogonality, orthogonality preserving
mappings, linear operators, minimum modulus}

\maketitle

\bigskip

\section{\large{\bf{Introduction}}}

Suppose $H$ is a Hilbert spaces, and $\langle \cdot,\cdot \rangle$ is the inner
product on $H$. The usual orthogonality relation $\perp$ is defined by
\[
x\perp y\Leftrightarrow \langle x,y\rangle=0.
\]

A mapping $f$ : $H\rightarrow$ $H$ satisfying the condition: for every $x,y\in
H,$%
\[
x\perp y\Rightarrow f(x)\perp f(y)
\]
is called an \emph{orthogonality-preserving} (\emph{OP}) mapping.

Let $B\left(  H\right)  $ denote the set of all bounded linear operators on
$H$. It is well-known that a linear operator $T\in B(H)$ is OP if and only if
$T$ is a scalar multiple of an isometry.

Let us say that for a given $\varepsilon \in \lbrack0,1],$ two vectors $x,y\in
H$ are $\varepsilon$\emph{-orthogonal}, denoted by $x\perp^{\varepsilon}$ $y$
, if%

\[
\left \vert \langle x,y\rangle \right \vert \leq \varepsilon \Vert x\Vert \cdot \Vert
y\Vert.
\]

It is clear that every pair of vectors are $1$-orthogonal, so the interesting
case is when $\varepsilon \in \lbrack0,1).$

An operator $T\in B\left(  H\right)  $ is \emph{approximately
orthogonality-preserving }(AOP) if there is an $\varepsilon \in \lbrack0,1)$
such that, for every $x,y\in H,$%
\[
x\perp y\Rightarrow Tx\perp^{\varepsilon}Ty.
\]
If we want to include $\varepsilon$ in the notation, we say that $T$ is
$\varepsilon$-AOP. We say that every operator is $1$-AOP. If $0\leq
\varepsilon_{1}\leq \varepsilon_{2}<1$ and $T$ is $\varepsilon_{1}$-AOP, then
$T$ is $\varepsilon_{2}$-AOP. Thus we are interested in the smallest such
$\varepsilon.$

We define a function $\hat{\varepsilon}:B\left(  H\right)  \rightarrow \left[
0,1\right]  $ by
\[
\hat{\varepsilon}\left(  T\right)  =\inf \left \{  \varepsilon \in \lbrack
0,1]:T\text{ is }\varepsilon \text{-AOP }\right \}  .
\]
Thus $\hat{\varepsilon}\left(  T\right)  =1$ whenever $T$ is not AOP.

In \cite{Chm6}, stability property for inner product preserving (not
necessarily linear) mappings was studied. Other approximate orthogonalities in
general normed spaces along with the corresponding approximately
orthogonality-preserving mappings have been studied in \cite{Chm10}%
,\cite{Tur10},\cite{Wso},\cite{ZM}.

In \cite[Theorem 2]{Chm5}, Chmieli\'{n}ski proved every nonzero linear AOP
operator is bounded from below. In this paper we prove that the converse
holds, i.e., $T$ is AOP if and only if $T$ is bounded from below.

Recall that the \emph{minimum modulus} $m\left(  T\right)  $ of $T$ is defined
to be the largest number $m\geq0$ such that, for every $x\in H,$%
\[
\left \Vert Tx\right \Vert \geq m\left \Vert x\right \Vert .
\]
Our main result (Theorem \ref{3}) is the following.
\\
\\
{\bf Theorem 2.3.}$\ $
 Suppose $T\in B(H)\backslash \{0\}.$ Then
\[
\hat{\varepsilon}\left(  T\right)  =\frac{\left \Vert T\right \Vert
^{2}-m\left(  T\right)  ^{2}}{\left \Vert T\right \Vert ^{2}+m\left(  T\right)
^{2}}
\]
and
\[
m\left(  T\right)  =\sqrt{\frac{1-\hat{\varepsilon}\left(  T\right)  }%
{1+\hat{\varepsilon}\left(  T\right)  }}\left \Vert T\right \Vert .
\]
\\

Clearly, this implies that $m\left(  T\right)  >0$ if and only if
$\hat{\varepsilon}\left(  T\right)  <1,$ and $m\left(  T\right)  =\left \Vert
T\right \Vert $ if and only if $\hat{\varepsilon}\left(  T\right)  =0$ if and
only if $T$ is OP.

When $H$ is finite-dimensional, Chmieli\'{n}ski \cite{Chm6} proved that there
is a function $\delta:[0,1)\rightarrow \lbrack0,\infty)$ such that
$\lim_{\varepsilon \rightarrow0^{+}}\delta \left(  \varepsilon \right)  =0$ such
that if $0\leq \varepsilon<1$ and $T\in B\left(  H\right)  $ is $\varepsilon
$-AOP, then there is linear OP mapping $S$ such that
\[
\Vert T-S\Vert \leq \delta \left(  \varepsilon \right)  \min \left \{  \Vert
T\Vert,\left \Vert S\right \Vert \right \}  ,
\]
and asked whether the same holds true when $H$ is infinite-dimensional.

A. Turn\v{s}ek \cite[Theorem 2.3]{Tur} showed that Chmieli\'{n}ski's result
\cite{Chm6} holds for arbitrary $H$ with%
\[
\delta \left(  \varepsilon \right)  =1-\sqrt{\frac{1-\varepsilon}{1+\varepsilon
}},
\]
and he claimed, using an example \cite[Example 2.4]{Tur}, that his result is
sharp. However, we show that Chmieli\'{n}ski's result holds when%
\[
\delta \left(  \varepsilon \right)  =\frac{1-\sqrt{\frac{1-\varepsilon
}{1+\varepsilon}}}{1+\sqrt{\frac{1-\varepsilon}{1+\varepsilon}}}.
\]
Thus, if $\hat{\varepsilon}\left(  T\right)  <1$, there is a linear OP map $S$
such that%
$$
\Vert T-S\Vert \leq \left(  \frac{1-\sqrt{\frac{1-\hat{\varepsilon}\left(
T\right)  }{1+\hat{\varepsilon}\left(  T\right)  }}}{1+\sqrt{\frac
{1-\hat{\varepsilon}\left(  T\right)  }{1+\hat{\varepsilon}\left(  T\right)
}}}\right)  \cdot \min \left \{  \Vert T\Vert,\left \Vert S\right \Vert \right \}
=\frac{1-\sqrt{\frac{1-\hat{\varepsilon}\left(  T\right)  }{1+\hat
{\varepsilon}\left(  T\right)  }}}{2}\left \Vert T\right \Vert .
$$
It follows from Theorem \ref{4} that $\delta \left(  \varepsilon \right)  $ is
the best. Note that if $T=V\left(  T^{\ast}T\right)  ^{1/2}$ is the polar
decomposition of $T,$ Turn\v{s}ek defines $S=\left \Vert T\right \Vert V$, while
we choose $S=\frac{\left \Vert T\right \Vert +m\left(  T\right)  }{2}V$.

Since linear OP mappings are precisely scalar multiples of isometries, a
natural question is whether linear $\varepsilon$-AOP mappings are close to
linear OP mappings (that is, to scalar multiples of isometries) as
$\varepsilon \rightarrow0$. In other words, does $\hat{\varepsilon}\left(
T\right)  $ in some way measure the distance from $T$ to the set
$\mathbb{C}\mathcal{V}$ of scalar multiples of isometries? We prove the
following affirmative answer (Theorem 3.5):%
\[
\hat{\varepsilon}\left(  T\right)  <1\Rightarrow \mathrm{dist}\left(
T,\mathbb{C}\mathcal{V}\right)  =\frac{1-\sqrt{\frac{1-\hat{\varepsilon
}\left(  T\right)  }{1+\hat{\varepsilon}\left(  T\right)  }}}{2}\left \Vert
T\right \Vert .
\]
When $H$ is separable, we actually prove that this formula holds for all
operators $T$ that are not semi-Fredholm with positive index.

If we replace the set of scalar multiples of isometries with the set
$\mathbb{C}\mathcal{U}$ of scalar multiples of the set $\mathcal{U}$ of
unitary operators, we obtain a universal distance formula on a separable
Hilbert space:%
\[
\mathrm{dist}\left(  T,\mathbb{C}\mathcal{U}\right)  =\left \{
\begin{array}
[c]{cc}%
\frac{\left \Vert T\right \Vert +m_{e}\left(  T^{\ast}\right)  }{2} & \text{if
}\dim \ker T>\dim \ker T^{\ast}\text{ }\\
\frac{\left \Vert T\right \Vert +m_{e}\left(  T\right)  }{2} & \text{if }%
\dim \ker T<\dim \ker T^{\ast}\\
\frac{1-\sqrt{\frac{1-\hat{\varepsilon}\left(  T\right)  }{1+\hat{\varepsilon
}\left(  T\right)  }}}{2}\left \Vert T\right \Vert  & \text{if }\dim \ker
T=\dim \ker T^{\ast}.
\end{array}
\right.
\]

When $\dim H=\aleph_{0}$, the formula for $\mathrm{dist}\left(  T,\mathbb{C}%
\mathcal{U}\right)  $ cannot be written solely in terms of $\left \Vert
T\right \Vert $ and $\hat{\varepsilon}\left(  T\right)  $; it seems likely that
the same is true for a formula for $dist\left(  T,\mathbb{C}\mathcal{V}%
\right)  $.

\section{\large{\bf{Main Results}}}

In this paper, we assume the dimension of $H$ is at least $2$.

The following lemma is a well-known result about left invertible operators.

\begin{lemma}
Suppose $T\in B(H).$ The following are equivalent.
\end{lemma}

$(1)\ m(T):=\inf \{ \Vert Tx\Vert:\Vert x\Vert=1\}>0$ (i.e., $T$ is bounded from
below)$.$

$(2)$ $T$ is left invertible.

$(3)$ $\mathrm{\ker}T=0$ and $\mathrm{ran}T$ is closed.

\begin{lemma}
Suppose $T\in B(H)\backslash \{0\}$ and $m(T)=0.$ Then $\hat{\varepsilon}(T)=1.$
\end{lemma}

\begin{proof}
Suppose $\ker T\neq0.$ Take any nonzero vector $e\in \ker T\ $and any unit
vector $f\notin \ker T\cup$ $\ker T^{\perp}.$ Let $x=$ $\langle e,f\rangle f,$
$y=e-x.$ Then $x\perp y$ and $Tx=-Ty,$ thus it follows from Cauchy-Schwarz
inequality that $\hat{\varepsilon}(T)=1.$

Now, we may assume $\ker T=0,$ but
ran $T$ is not closed. Let $E$ be the spectral measure of $|T|.$ Then for
every $\delta>0,$ $E[0,\delta)H$ is an infinite dimensional closed subspace
such that $\Vert Tx\Vert<\delta \Vert x\Vert$ for all $x\in$ $E[0,\delta]H.$
Let $\delta_{n}=\frac{\Vert T\Vert}{2n},\ M_{n}=E[0,\delta_{n})H,\ n\in%
%TCIMACRO{\U{2115} }%
%BeginExpansion
\mathbb{N}
%EndExpansion
.$ Take unit vectors $e_{1}\in M_{1}$ and $f_{1}$ $\perp e_{1}.$ For $n\geq2,$
take unit vectors $e_{n}\in M_{n}\cap \{e_{1},e_{2},...,e_{n-1}\}^{\perp}$,
$f_{n}$ $\perp e_{n}$ with $\Vert Tf_{n}\Vert=\frac{2n-1}{2n}\Vert T\Vert.$
Let $x_{n}=e_{n}-f_{n},$ $y_{n}=e_{n}+f_{n}.$ Then $x_{n}\perp y_{n}$ , $\Vert
Tx_{n}\Vert \rightarrow \Vert T\Vert,$ $\Vert Ty_{n}\Vert \rightarrow \Vert
T\Vert,$ and $|\langle Tx_{n},Ty_{n}\rangle|\rightarrow \Vert T\Vert^{2},$ this
shows that $\hat{\varepsilon}(T)=1$ and the proof is completed.
\end{proof}

\begin{theorem}
\label{3}Suppose $T\in B(H)\backslash \{0\}$. Then
\[
\hat{\varepsilon}(T)=\frac{\left \Vert T\right \Vert ^{2}-m\left(  T\right)
^{2}}{\left \Vert T\right \Vert ^{2}+m\left(  T\right)  ^{2}}.
\]

\end{theorem}

\begin{proof}
Let $m=m(T),t=\Vert T\Vert.$ By the preceding lemma, we can easily see that
the equation holds when $m=0$.

Now, let's assume $m(T)>0.$ Notice that rank$T\geq2.$ Take two unit vectors
$Th$, $Tk$ with $Th\perp Tk$, then $h$ and $k$ are linearly independent, and
$\frac{1}{t}\leq \Vert h\Vert,\Vert k\Vert \leq \frac{1}{m}.$ Suppose $\langle
h,k\rangle=re^{i\theta},$ and $\lambda=\frac{\Vert h\Vert}{\Vert k\Vert
}e^{i\theta}.$ Then
\[
|\lambda|\in \lbrack m/t,t/m],\  \langle h,\lambda k\rangle \in%
%TCIMACRO{\U{211d} }%
%BeginExpansion
\mathbb{R}
%EndExpansion
\  \text{and }h+\lambda k\perp h-\lambda k.
\]
We compute
\[
\langle T(h+\lambda k),T(h-\lambda k)\rangle=1-|\lambda|^{2},
\]%
\[
\Vert T(h+\lambda k)\Vert \Vert T(h-\lambda k)\Vert=1+|\lambda|^{2},
\]
so
\[
\hat{\varepsilon}(T)\geq{\sup}\left \{  \frac{|1-|\lambda|^{2}|}{1+|\lambda
|^{2}}:|\lambda|\in \lbrack m/t,t/m]\right \}  =\frac{t^{2}-m^{2}}{t^{2}+m^{2}}.
\]

Suppose $x\perp y_{0}$, $\Vert x\Vert=\Vert y_{0}\Vert=1$ and $\varepsilon
_{xy_{0}}=\langle \frac{Tx}{\Vert Tx\Vert},\frac{Ty_{0}}{\Vert Ty_{0}\Vert
}\rangle=|\varepsilon_{xy_{0}}|e^{i\theta}.$ Let $y=e^{i\theta}y_{0}.$ Then
$\varepsilon_{xy}=\langle \frac{Tx}{\Vert Tx\Vert},\frac{Ty}{\Vert Ty\Vert
}\rangle=|\varepsilon_{xy_{0}}|\  \geq0$ and
\begin{align*}
2-2\varepsilon_{xy} & =\Vert \frac{Tx}{\Vert Tx\Vert}-\frac{Ty}{\Vert Ty\Vert
}\Vert^{2}=\Vert T(\frac{x}{\Vert Tx\Vert}-\frac{y}{\Vert Ty\Vert})\Vert^{2}\\
&  \geq m^{2}\Vert \frac{x}{\Vert Tx\Vert}-\frac{y}{\Vert Ty\Vert}\Vert^{2}\\
& =m^{2}(\Vert Tx\Vert^{-2}+\Vert Ty\Vert^{-2}).
\end{align*}
Suppose $\Vert \frac{Tx}{\Vert Tx\Vert}-\frac{Ty}{\Vert Ty\Vert}\Vert
^{2}=\lambda m^{2}(\Vert Tx\Vert^{-2}+\Vert Ty\Vert^{-2})$ with $1\leq$
$\lambda \leq$ $\frac{t^{2}}{m^{2}}$. By the Parallelogram law,%
\[
\Vert \frac{Tx}{\Vert Tx\Vert}+\frac{Ty}{\Vert Ty\Vert}\Vert^{2}=4-\lambda
m^{2}(\Vert Tx\Vert^{-2}+\Vert Ty\Vert^{-2}).
\]
Notice that $\Vert \frac{Tx}{\Vert Tx\Vert}+\frac{Ty}{\Vert Ty\Vert}\Vert
^{2}=\Vert T(\frac{x}{\Vert Tx\Vert}+\frac{y}{\Vert Ty\Vert})\Vert^{2}\leq$
$t^{2}(\Vert Tx\Vert^{-2}+\Vert Ty\Vert^{-2}),$ we get
\[
4\leq(t^{2}+\lambda m^{2})(\Vert Tx\Vert^{-2}+\Vert Ty\Vert^{-2}),
\]
thus $$\Vert Tx\Vert^{-2}+\Vert Ty\Vert^{-2}\geq \frac{4}{t^{2}+\lambda m^{2}}.$$
Since $\lambda \geq1,$
\begin{align*}
2-2\varepsilon_{xy} & =\Vert \frac{Tx}{\Vert Tx\Vert}-\frac{Ty}{\Vert Ty\Vert
}\Vert^{2}=\lambda m^{2}(\Vert Tx\Vert^{-2}+\Vert Ty\Vert^{-2})\\
& =\geq \frac{4\lambda m^{2}}{t^{2}+\lambda m^{2}}\geq \frac{4m^{2}}{t^{2}%
+m^{2}},
\end{align*}
this proves $\varepsilon_{xy}\leq \frac{t^{2}-m^{2}}{t^{2}+m^{2}}$ and
therefore
\[
\hat{\varepsilon}(T)=\underset{x\perp y}{\sup}\varepsilon_{xy} \leq \frac
{t^{2}-m^{2}}{t^{2}+m^{2}}.
\]
The proof is completed.
\end{proof}

The following corollaries follow directly from the theorem.

\begin{corollary}
Suppose $T\in B(H).$ Then $T$ is an orthogonality preserving mapping if and
only if $T$ is a scalar multiple of an isometry.
\end{corollary}

\begin{corollary}
Suppose $T\in B(H)\backslash \{0\}.$ Then there exists an $\varepsilon
\in \lbrack0,1)$ such that $T$ is an $\varepsilon$-$OP$ mapping if and only if
$T$ is bounded from below$.$ Moreover, $m(T)=\sqrt{\frac{1-\hat{\varepsilon
}(T)}{1+\hat{\varepsilon}(T)}}\Vert T\Vert.$
\end{corollary}

\begin{corollary}
Suppose $T\in B(H)\backslash \{0\}.$ Then $\hat{\varepsilon}\ $is continuous at
$T.$
\end{corollary}

\begin{proof}
Suppose $\left \Vert T_{n}-T\right \Vert \rightarrow0.$ Since $T\neq0,$ we may
assume all the $T_{n}$ 's are not zero. Then $t_{n}=\Vert T_{n}\Vert
\  \rightarrow \Vert T\Vert=t$, $m_{n}=m(T_{n})\rightarrow m(T)=m,$ and
$t_{n}^{2}+m_{n}^{2}\neq0,$ therefore%
\[
\hat{\varepsilon}(T_{n})=\frac{t_{n}^{2}-m_{n}^{2}}{t_{n}^{2}+m_{n}^{2}%
}\rightarrow \frac{t^{2}-m^{2}}{t^{2}+m^{2}}=\hat{\varepsilon}(T).
\]

\end{proof}

\begin{remark}
The function $\hat{\varepsilon}\ $is not continuous at $0.$ Take any $T\ $with
$\hat{\varepsilon}(T)\neq0\  \ $and let $T_{n}=\frac{1}{n}T.$ Then $\Vert
T_{n}\Vert \rightarrow0,$ but for every $n,\  \hat{\varepsilon}(T_{n}%
)=\hat{\varepsilon}(T)\neq0.$
\end{remark}

\section{\large{\bf{A Distance Formula}}}

Let $\mathcal{V}$ be the set of all isometries, $%
%TCIMACRO{\U{2102} }%
%BeginExpansion
\mathbb{C}
%EndExpansion
\mathcal{V}$ be the set of all scalar multiples of isometries.
Since $$\hat{\varepsilon}(T)=0 \Leftrightarrow T\in
\mathbb{C}
\mathcal{V\Leftrightarrow} \ \textrm{dist}(T,
\mathbb{C}
\mathcal{V})=0,$$ a natural question comes up. Suppose $\{T_{n}\}_{n=1}%
^{\infty}\subseteq B(H)$ and $\hat{\varepsilon}\left(  T_{n}\right)
\rightarrow0.$ Does that imply \textrm{dist}$(T_{n},%
%TCIMACRO{\U{2102} }%
%BeginExpansion
\mathbb{C}
%EndExpansion
\mathcal{V})\rightarrow0?$ Unfortunately, the answer is negative.

\begin{example}
\label{example}For each $n\in%
%TCIMACRO{\U{2115} }%
%BeginExpansion
\mathbb{N}
%EndExpansion
,$ let $T_{n}=\left(
\begin{array}
[c]{cc}%
n^{2} & 0\\
0 & n^{2}+n
\end{array}
\right)  \in M_{2}(%
%TCIMACRO{\U{211d} }%
%BeginExpansion
\mathbb{R}
%EndExpansion
).$ Clearly, $$\hat{\varepsilon}\left(  T_{n}\right)  =\hat{\varepsilon}%
(\frac{T_{n}}{n^{2}})=\frac{(1+\frac{1}{n})^{2}-1}{(1+\frac{1}{n})^{2}%
+1}\rightarrow0.$$ We claim \textrm{dist}$(T_{n},%
%TCIMACRO{\U{2102} }%
%BeginExpansion
\mathbb{C}
%EndExpansion
\mathcal{V})=\frac{n}{2}\rightarrow \infty.$ To see this, take an
isometry(unitary) matrix $U=\left(
\begin{array}
[c]{cc}%
a & b\\
c & d
\end{array}
\right)  \in M_{2}(%
%TCIMACRO{\U{211d} }%
%BeginExpansion
\mathbb{R}
%EndExpansion
),\lambda \in%
%TCIMACRO{\U{211d} }%
%BeginExpansion
\mathbb{R}
%EndExpansion
.$ Then
\begin{align*}
\Vert T_{n}-\lambda U\Vert & =\Vert \left(
\begin{array}
[c]{cc}%
n^{2}-\lambda a & -\lambda b\\
-\lambda c & (n^{2}+n)-\lambda d\\
&
\end{array}
\right)  \Vert \\
& \geq \max \{|n^{2}-\lambda a|,\ |(n^{2}+n)-\lambda d|\}.
\end{align*}
Since $U$ is an unitary, $|a|=|d|.$ If $ad<0,$ then
\[
\max \{|n^{2}-\lambda a|,\ |(n^{2}+n)-\lambda d|\} \geq n^{2}.
\]
If $ad>0,$ then $\lambda a=\lambda d$ and
\[
\max \{|n^{2}-\lambda a|,\ |(n^{2}+n)-\lambda d|\} \geq \frac{n}{2}.
\]
Consequently, $\mbox{dist}(T_{n},%
%TCIMACRO{\U{2102} }%
%BeginExpansion
\mathbb{C}
%EndExpansion
\mathcal{V})\geq \frac{n}{2}.$ But $\Vert T_{n}-(n^{2}+\frac{n}{2})I\Vert
=\frac{n}{2},$ so $\mbox{dist} (T_{n},%
%TCIMACRO{\U{2102} }%
%BeginExpansion
\mathbb{C}
%EndExpansion
\mathcal{V})=\frac{n}{2}.$
\end{example}

The above example gives us a way to compute the distance between some special
operators and $%
%TCIMACRO{\U{2102} }%
%BeginExpansion
\mathbb{C}
%EndExpansion
\mathcal{V}$. In the following, we study the general distance formula.

\begin{theorem}
\bigskip Let $T$\ $\in B(H).$ Then
\[
\mathrm{dist}(T,%
%TCIMACRO{\U{2102} }%
%BeginExpansion
\mathbb{C}
%EndExpansion
\mathcal{V})\geq \frac{\Vert T\Vert-m(T)}{2}.
\]

\end{theorem}

\begin{proof}
Let \ $\lambda \in%
%TCIMACRO{\U{2102} }%
%BeginExpansion
\mathbb{C}
%EndExpansion
,V$\ $\in \mathcal{V},\ x\in H\ $with $\Vert x\Vert=1.\ $Then \  \
\[
\Vert T-\lambda V\Vert \  \geq \Vert Tx-\lambda Vx\Vert \  \geq|\Vert
Tx\Vert-|\lambda||,\
\]
so
\[
\Vert T-\lambda V\Vert \text{ }\geq \underset{\Vert x\Vert=1}{\sup}|\  \Vert
Tx\Vert-|\lambda|\ |\text{ }\geq \max \{ \Vert T\Vert-|\lambda|,|\lambda
|-m(T)\},
\]
therefore%
\[
\mathrm{dist}(T,%
%TCIMACRO{\U{2102} }%
%BeginExpansion
\mathbb{C}
%EndExpansion
\mathcal{V})=\underset{\lambda V\in%
%TCIMACRO{\U{2102} }%
%BeginExpansion
\mathbb{C}
%EndExpansion
\mathcal{V}}{\inf}\Vert T-\lambda V\Vert \text{ }\geq \underset{\lambda \in%
%TCIMACRO{\U{2102} }%
%BeginExpansion
\mathbb{C}
%EndExpansion
}{\text{ }\inf}\max \{ \Vert T\Vert-|\lambda|,|\lambda|-m(T)\}=\frac{\Vert
T\Vert-m(T)}{2}.
\]

\end{proof}

\begin{theorem}
Suppose $T\in B(H)$ and $\dim \ker T\leq \dim \ker T^{\ast}.$ Then
\[
\mathrm{dist}(T,%
%TCIMACRO{\U{2102} }%
%BeginExpansion
\mathbb{C}
%EndExpansion
\mathcal{V})\leq \frac{\Vert T\Vert-m(T)}{2}.
\]

\end{theorem}

\begin{proof}
Since $\dim \ker T\leq \dim \ker T^{\ast},$ $T=V|T|$ for some isometry $V.$
Let $\lambda=\frac{\Vert T\Vert+m(T)}{2}.$ Then
\[
\Vert T-\lambda V\Vert=\Vert V(|T|-\lambda)\Vert=\Vert|T|-\lambda \Vert
= \frac{\Vert T\Vert-m(T)}{2}.
\]

\end{proof}

\begin{remark}
If we rewrite $\frac{\Vert T\Vert-m(T)}{2}=\frac{\Vert T\Vert-m(T)}{\Vert
T\Vert+m(T)}\cdot \frac{\Vert T\Vert+m(T)}{2},$ then it follows from proof of
the preceding theorem that
\[
\Vert T-\lambda V\Vert=\frac{\Vert T\Vert-m(T)}{2}=\frac{\Vert T\Vert
-m(T)}{\Vert T\Vert+m(T)}\cdot\frac{\Vert T\Vert+m(T)}{2}%
\]%
\[
=\frac{1-\sqrt{\frac{1-\hat{\varepsilon}(T)}{1+\hat{\varepsilon}(T)}}}%
{1+\sqrt{\frac{1-\hat{\varepsilon}(T)}{1+\hat{\varepsilon}(T)}}}\cdot\min \{ \Vert
T\Vert,\Vert \lambda V\Vert \},
\]
and clearly, $$\frac{1-\sqrt{\frac{1-\hat{\varepsilon}(T)}{1+\hat{\varepsilon
}(T)}}}{1+\sqrt{\frac{1-\hat{\varepsilon}(T)}{1+\hat{\varepsilon}(T)}}%
}<1-\sqrt{\frac{1-\hat{\varepsilon}(T)}{1+\hat{\varepsilon}(T)}},$$ this is a
sharper result than \cite[Theorem 2.3]{Tur}.
\end{remark}

\begin{theorem}
\label{4}Let $T\in B(H).$ If $\dim \ker T\leq \dim \ker T^{\ast},$ then
\[
\mathrm{dist}(T,%
%TCIMACRO{\U{2102} }%
%BeginExpansion
\mathbb{C}
%EndExpansion
\mathcal{V})=\frac{\Vert T\Vert-m(T)}{2}=\frac{1-\sqrt{\frac{1-\hat
{\varepsilon}(T)}{1+\hat{\varepsilon}(T)}}}{2}\Vert T\Vert.
\]
\end{theorem}
\begin{proof}
Combine Theorems 3.2 and 3.3.

\end{proof}

Recall that \textrm{ind}$T=\dim \ker T-\dim \ker T^{\ast}.$ If $H$ is separable, it
is known that the closure of the set of $T\in B(H)$ with \textrm{ind}%
$T\leq0$ is the complement of the set of \emph{semi-Fredholm operators} $T$
with \textrm{ind}$T>0$ (i.e., $\mathrm{ran}T$ is closed, $\dim \ker T^{\ast
}<\infty \ $and $\dim \ker T>\dim \ker T^{\ast}$). Since the distance function is
continuous, we get the following corollary.

\begin{corollary}
Suppose $H$ is separable and $T\in B(H).$ If \textrm{ran}$T$ is not closed or
\textrm{ind}$T\leq0,$ then
\[
\mathrm{dist}(T,%
%TCIMACRO{\U{2102} }%
%BeginExpansion
\mathbb{C}
%EndExpansion
\mathcal{V})=\frac{\Vert T\Vert-m(T)}{2}=\frac{1-\sqrt{\frac{1-\hat
{\varepsilon}(T)}{1+\hat{\varepsilon}(T)}}}{2}\Vert T\Vert.
\]

\end{corollary}

\begin{remark}
In this case, if \textrm{ran}$T$ is not closed, then $m(T)=0;$ hence
$$\mbox{dist}(T,
\mathbb{C}
\mathcal{V})=\frac{\Vert T\Vert}{2}.$$
\end{remark}

\begin{proposition}
Suppose $T\in B(H)$ is right invertible but not invertible. Then
\[
0<m(T^{\ast})\leq \mathrm{dist}(T,%
%TCIMACRO{\U{2102} }%
%BeginExpansion
\mathbb{C}
%EndExpansion
\mathcal{V})\leq \Vert T\Vert.
\]

\end{proposition}

\begin{proof}
Suppose $TS=1$ for some $S\in B(H).$ Then $\Vert S\Vert^{-1}=m(T^{\ast}).$ Let
$\lambda V\in%
%TCIMACRO{\U{2102} }%
%BeginExpansion
\mathbb{C}
%EndExpansion
\mathcal{V}$. Then
\[
\Vert T-\lambda V\Vert \geq \Vert1-\lambda VS\Vert \cdot \Vert S\Vert^{-1}%
=\Vert1-\lambda VS\Vert \cdot \ m(T^{\ast}).
\]
If $\Vert1-\lambda VS\Vert \ <1,$ we see that $VS$ is invertible, so $V$ is
invertible and therefore $S$ is invertible, impossible. Clearly,
$\mathrm{dist}(T,%
%TCIMACRO{\U{2102} }%
%BeginExpansion
\mathbb{C}
%EndExpansion
\mathcal{V})\leq \Vert T\Vert.$ The proof is complete.
\end{proof}

\begin{remark}
If $T$ is a scalar multiple of a nonunitary co-isometry, then $m(T^{\ast
})=\Vert T\Vert$ and therefore $\mathrm{dist}(T,%
%TCIMACRO{\U{2102} }%
%BeginExpansion
\mathbb{C}
%EndExpansion
\mathcal{V})=\Vert T\Vert.$
\end{remark}

Even Example \ref{example} tells us $\mathrm{dist}(T_{n},%
%TCIMACRO{\U{2102} }%
%BeginExpansion
\mathbb{C}
%EndExpansion
\mathcal{V})$ may not converge to $0$ in the case that $\hat{\varepsilon
}\left(  T_{n}\right)  \rightarrow0,$ we still have the following result.

\begin{theorem}
Suppose $\{T_{n}\}_{n=1}^{\infty}\subseteq B(H)\backslash \{0\}$ and
$\hat{\varepsilon}\left(  T_{n}\right)  \rightarrow0.$ Then
\[
\mathrm{dist}(\frac{T_{n}}{\Vert T_{n}\Vert},\mathcal{V})\rightarrow0.
\]

\end{theorem}

\begin{proof}
We may assume all the $\hat{\varepsilon}\left(  T_{n}\right)  $'s are smaller
than 1. For each $n,$ let $t_{n}^{{}}=\Vert T_{n}\Vert$ and $m_{n}=m(T_{n}).$
Since $$\hat{\varepsilon}\left(  T_{n}\right)  =\frac{t_{n}^{2}-m_{n}^{2}}%
{t_{n}^{2}+m_{n}^{2}}=\frac{1-\frac{m_{n}^{2}}{t_{n}^{2}}}{1+\frac{m_{n}^{2}%
}{t_{n}^{2}}}\rightarrow0,$$ it follow from Theorem 2.3 that $\frac{m_{n}^{{}}}{t_{n}^{{}}%
}\rightarrow1.$ Let $T_{n}=V_{n}|T_{n}|$ be the polar decomposition of each
$T_{n}$. We see
\[
\Vert \frac{T_{n}}{\Vert T_{n}\Vert}-V_{n}\Vert=\Vert \frac{|T_{n}|}{\Vert
T_{n}\Vert}-I\Vert \leq1-\frac{m_{n}^{{}}}{t_{n}^{{}}}\rightarrow0,
\]
$\ $ and therefore $\mathrm{dist}(\frac{T_{n}}{\Vert T_{n}\Vert},\mathcal{V}%
)\rightarrow0.$
\end{proof}

\begin{remark}
We keep the same notations as in the proof of the above theorem. Actually,
for all $\lambda_{n}\in \lbrack m_{n},t_{n}],$ $\mathrm{dist}(\frac{T_{n}%
}{\lambda_{n}},\mathcal{V})\rightarrow0.$ This is because for each
$\lambda_{n}\in \lbrack m_{n},t_{n}],$%
\[
\Vert T_{n}-\lambda_{n}V_{n}\Vert=\Vert|T_{n}|-\lambda_{n}\Vert \leq
t_{n}-m_{n},
\]
so we have
\[
\Vert \frac{T_{n}}{\lambda_{n}}-V_{n}\Vert \leq \frac{t_{n}-m_{n}}{\lambda_{n}%
}\leq \frac{t_{n}}{m_{n}}-1\rightarrow0.
\]

\end{remark}

\begin{proposition}
Suppose $S,T\in B(H).$ Then

$(1)$ If $S,T\neq0,\ $then $$\hat{\varepsilon}(ST)\leq \frac{t_{1}^{2}t_{2}%
^{2}-m_{1}^{2}m_{2}^{2}}{t_{1}^{2}t_{2}^{2}+m_{1}^{2}m_{2}^{2}},$$ where $\Vert
S\Vert=t_{1,}$ $m(S)=m_{1,}$ $\Vert T\Vert=t_{2},$ $m(T)=m_{2}.$

$(2)$ If $0\neq S$ $\in$ $%
%TCIMACRO{\U{2102} }%
%BeginExpansion
\mathbb{C}
%EndExpansion
\mathcal{V},$ then $\hat{\varepsilon}(ST)=\hat{\varepsilon}(T).$

$(3)$ If $T\in$ $%
%TCIMACRO{\U{2102} }%
%BeginExpansion
\mathbb{C}
%EndExpansion
\mathcal{V},$ then $\hat{\varepsilon}(ST)\leq \hat{\varepsilon}(S).$
\end{proposition}
\begin{proof}
(1) Since $m(S)m(T)$ $\leq m(ST)\leq \Vert ST\Vert \leq \Vert S\Vert \cdot \Vert
T\Vert$, it is easy to see that
\[
\hat{\varepsilon}(ST)=\frac{\Vert ST\Vert^{2}-m(ST)^{2}}{\Vert ST\Vert
^{2}+m(ST)^{2}}\leq \frac{t_{1}^{2}t_{2}^{2}-m_{1}^{2}m_{2}^{2}}{t_{1}^{2}%
t_{2}^{2}+m_{1}^{2}m_{2}^{2}}.
\]

(2) This is because $\Vert ST\Vert=\Vert S\Vert \cdot \Vert T\Vert \ $and
$m(S)m(T)$ $=m(ST)$ if $S$ $\in$ $%
%TCIMACRO{\U{2102} }%
%BeginExpansion
\mathbb{C}
%EndExpansion
\mathcal{V}$.

(3) Let $T\in$ $%
%TCIMACRO{\U{2102} }%
%BeginExpansion
\mathbb{C}
%EndExpansion
\mathcal{V}.$ If $T=0,$ then $\hat{\varepsilon}(ST)=0\leq$ $\hat
{\varepsilon}(S).$ If $T\neq0,$ then $$t_{2}=\Vert T\Vert=m(T)=m_{2},$$ and
hence
\[
\hat{\varepsilon}(ST)\leq \frac{t_{1}^{2}t_{2}^{2}-m_{1}^{2}m_{2}^{2}}%
{t_{1}^{2}t_{2}^{2}+m_{1}^{2}m_{2}^{2}}=\frac{t_{1}^{2}-m_{1}^{2}}{t_{1}%
^{2}+m_{1}^{2}}=\hat{\varepsilon}(S).
\]

\end{proof}

We know from the above proposition that $\hat{\varepsilon}(ST)\leq
\hat{\varepsilon}(S)$ if $T\in$ $%
%TCIMACRO{\U{2102} }%
%BeginExpansion
\mathbb{C}
%EndExpansion
\mathcal{V}$. Moreover, in this case, $\hat{\varepsilon}(ST)$ might take any
value smaller than $\hat{\varepsilon}(S),$ see the following example.

\begin{example}
Suppose $H=\ell^{2}.$ Given $\lambda \in%
%TCIMACRO{\U{2102} }%
%BeginExpansion
\mathbb{C}
%EndExpansion
,\  \delta \in \lbrack0,1]$ define two linear operators $T$ by
\[
Tx=(0,0,\lambda x_{1},\lambda x_{2},...),
\]
$S$ by
\[
Sx=(x_{1},2x_{2},(1+\delta)x_{3},(2-\delta)x_{4},(1+\delta)x_{5}%
,(2-\delta)x_{6},...),
\] where $x=(x_{n})\in H.$
Then $T\in$ $%
%TCIMACRO{\U{2102} }%
%BeginExpansion
\mathbb{C}
%EndExpansion
\mathcal{V}$ and $\Vert T\Vert=|\lambda|$,
\[
STx=(0,0,(1+\delta)\lambda x_{1},(2-\delta)\lambda x_{2},(1+\delta)\lambda
x_{3},(2-\delta)\lambda x_{4},...).
\]
Clearly, $\lambda=0\ $implies $\hat{\varepsilon}(ST)=0.$ If $\lambda \neq0,$ by
letting $\delta \rightarrow0,$ we see $\hat{\varepsilon}(ST)$ is converging to
$\frac{3}{5}=\hat{\varepsilon}(S),$ and letting $\delta \rightarrow \frac{1}%
{2},$ we see $\hat{\varepsilon}(ST)\rightarrow0.$
\end{example}

It is disappointing that we couldn't find a formula for $\mathrm{dist}(T,%
%TCIMACRO{\U{2102} }%
%BeginExpansion
\mathbb{C}
%EndExpansion
\mathcal{V})$ for arbitrary operator $T$, but note that if $H$ is finite
dimensional, isometries are unitaries, then $\mathrm{dist}(T,%
%TCIMACRO{\U{2102} }%
%BeginExpansion
\mathbb{C}
%EndExpansion
\mathcal{V})=\mathrm{dist}(T,%
%TCIMACRO{\U{2102} }%
%BeginExpansion
\mathbb{C}
%EndExpansion
\mathcal{U}).$ In \cite[Theorem 1.3]{Rog}, D. Rogers gave a formula for
distance to the unitaries when $H$ is separable. By using his result, we get a
formula for $\mathrm{dist}(T,%
%TCIMACRO{\U{2102} }%
%BeginExpansion
\mathbb{C}
%EndExpansion
\mathcal{U}),$ we end this paper by listing this formula. The proof is not
hard and left to reader.

\begin{theorem}
Suppose $H\ $is separable and $T\in B(H).$ Then
\end{theorem}

\[
\mbox{dist}\left(  T,\mathbb{C}\mathcal{U}\right)  =\left \{
\begin{array}
[c]{cc}%
\frac{\left \Vert T\right \Vert +m_{e}\left(  T^{\ast}\right)  }{2} & \text{if
}\dim \ker T>\dim \ker T^{\ast}\text{ }\\
\frac{\left \Vert T\right \Vert +m_{e}\left(  T\right)  }{2} & \text{if }%
\dim \ker T<\dim \ker T^{\ast}\\
\frac{\left \Vert T\right \Vert -m\left(  T\right)  }{2} & \text{if }\dim \ker
T=\dim \ker T^{\ast},
\end{array}
\right.
\]
$where$ $m_{e}\left(  T\right)  =\inf \{ \lambda:\lambda \in \sigma_{e}(|T|)\}.$
\newline \textbf{{Acknowledgement}} The authors would like to thank Professor
Junhao Shen for some valuable suggestions.

\end{document}